\documentclass[11 pt]{amsart}
\usepackage{latexsym,amsmath,amsfonts,graphicx}
\usepackage{subfigure}
\usepackage{bbm}
\usepackage{amssymb}
\usepackage{float}
\usepackage[dvips]{color}
\newtheorem{theorem}{Theorem}[section]
\newtheorem{thm}[theorem]{Theorem}
\newtheorem{prop}{Proposition}[section]

\newtheorem{lemma}[prop]{Lemma}
\newtheorem{remark}[prop]{Remark}
\newtheorem{defi}[prop]{Definition}

\newtheorem{example}{Example}[section]
\numberwithin{equation}{section}

\def\cal{\mathcal }

\def\Z{\mathbb Z}
\def\Q{\mathbb Q}

\def\mathscr{\mathcal }

\def\diag{\text{diag}}

\newcommand{\ba}{{\mathbf{a}}}

\newcommand{\bi}{{\mathbf{i}}}
\newcommand{\bj}{{\mathbf{j}}}

\newcommand{\bu}{{\mathbf{u}}}
\newcommand{\bv}{{\mathbf{v}}}

\newcommand{\bx}{{\mathbf{x}}}
\newcommand{\by}{{\mathbf{y}}}
\newcommand{\bz}{{\mathbf{z}}}

\newcommand{\SD}{{\mathcal D}}
\newcommand{\SE}{{\mathcal E}}

\begin{document}

\title[Lipschitz classification of  Bedford-McMullen carpets]
{Lipschitz classification of  Bedford-McMullen carpets with uniform horizontal fibers}


\author{Ya-min Yang} \address{Institute of applied mathematics, College of Science, Huazhong Agriculture University, Wuhan,430070, China.
} \email{yangym09@mail.hzau.edu.cn
 }
\author{Yuan Zhang$^*$} \address{Department of Mathematics and Statistics, Central China Normal University, Wuhan, 430079, China
} \email{yzhang@mail.ccnu.edu.cn
 }

\date{\today}
\thanks {The work is supported by NSFS Nos. 11971195 and 11601172.}


\thanks{* The correspondence author.}
\begin{abstract}
Let ${\cal M}_{t,v,r}(n,m)$, $2\leq m<n$,  be the collection of self-affine carpets with expanding matrix
$\diag(n,m)$ which are totally disconnected, possessing vacant rows and with uniform horizontal fibers.
In this paper, we introduce a notion of structure tree of a metric space, and thanks to this new notion, we completely characterize when two carpets in ${\cal M}_{t,v,r}(n,m)$
are   Lipschitz  equivalent.

\end{abstract}
\maketitle


 \section{\textbf{Introduction}}
Let $2\leq m<n$ be two  integers and denote by
$\diag (n,m)$ the diagonal matrix with entries $n$ and $m$.
 Let ${\cal D}\subset\{0,1,\dots,n-1\}\times\{0,1,\dots,m-1\}$  and we call it the \emph{digit set}.
 For $d\in \SD$, define
 $$S_d(z)=\diag(n^{-1}, m^{-1})(z+d).$$
Then $\{S_d\}_{d\in \SD}$ is an iterated function system (IFS), and  there exists a unique non-empty compact set $E=K(n,m,{\cal D})$
such that $E=\underset{d\in{\cal D}}\bigcup  S_d(E);$
we call  $E$  a \emph{Bedford-McMullen} carpet, or a \emph{self-affine carpet}.

Two metric spaces $(X, d_X)$ and $(Y, d_Y)$ are said to be Lipschitz equivalent,
denoted by $(X, d_X)\sim (Y, d_Y)$,
if there exists a map
$f:~X\rightarrow Y$ which is bi-Lipschitz, that is, there is a constant $C>0$ such that
$$
C^{-1}d_X(x,y)\leq d_Y(f(x),f(y))\leq C d_X(x,y), \quad \text{ for all } x,y\in X.
$$
There are many works on Lipschitz equivalence of self-similar sets, see \textbf{\cite{DS, FaMa92,LL13,  RRW12, RRX06,  R10,XX10}.}
For example, Rao, Ruan, and Xi \cite{RRX06} and Xi and Xiong \cite{XX10} showed that   for fractal cubes which are totally disconnected and have the same expanding matrix, the Hausdorff dimension completely determines the Lipschitz class. However, there are  few works  on the classification of self-affine carpets (\cite{Miao2013, Miao2017, Rao2019}).

The  study of the Lipschitz classification of totally disconnected self-affine carpets
 is much more difficult than that about the self-similar sets.
    In what follows, we use ${\mathcal M}_t(n,m)$  to  denote the collection of totally disconnected self-affine carpets  with expanding matrix $\diag(n,m)$.

Miao, Xi and Xiong\cite{Miao2017} and Rao, Yang and Zhang \cite{Rao2019} developed several  Lipschitz invariants of self-affine carpets which are very useful.
First, let us introduce some notations. Let $\# A$ be the cardinality of $A$.
 Let $E=K(n,m,\SD)$ be a self-affine carpet.
We define
\begin{equation}
a_j=\#\{i;~(i,j)\in {\cal D}\}, \quad 0\leq j\leq m-1,
\end{equation}
 and call $\ba=(a_j)_{j=0}^{m-1}$
the \emph{distribution sequence} of ${\cal D}$, or of $E$.
For a digit set $\SD$, we say the $j$-th row is vacant if $a_j=0$.
Miao \textit{et al.} \cite{Miao2017}  showed that
if  $E, F\in {\mathcal M}_t(n,m)$  are Lipschitz equivalent,
then either both of them possess vacant rows or neither  of them do.

Denote $N=\#{\SD}$. To remove the trivial case, we will always assume that $N>1$.
According to \cite{Hut81}, there is a unique Borel probability measure $\mu_E$ supported on $E$  satisfying
\begin{equation}\label{eq:Bernoulli}
\mu_E(\cdot)=\frac{1}{N}\sum_{d\in \SD} \mu_E\circ S_d^{-1}(\cdot),
\end{equation}
 and it is called   the \emph{uniform Bernoulli measure} of $E$.
Rao \textit{et al.} \cite{Rao2019}
 found several Lipschitz invariants related to the uniform Bernoulli measure of self-affine carpets.
 They prove that if $E, F\in {\cal M}_t(n,m)$ and $f:~E\to F$ is a bi-Lipschitz map, then $\mu_E$ and $\mu_F\circ f$ are equivalent; consequently,
 $\mu_E$ and $\mu_F$ have the same multifractal spectrum,
and  $\mu_E$ is doubling   if and only if $\mu_F$ is doubling.

 \begin{remark}\label{rem:doubling}{\rm (i) The mulitfractal spectrum of self-affine carpets have been completely characterized, see \cite{King95, Bar07, JR11}.

(ii)  A   measure $\nu$ on a metric space $X$
is said to be \emph{doubling} if there is a constant $C\geq 1$ such that
 $0<\nu(B(x,2r))\leq C\nu(B(x,r))<\infty$
 for all balls $B(x,r)\subset X$ with center $x$ and radius $r>0$.
According to Li, Wei and Wen \cite{LWW16},   $\mu_E$  is doubling
 if and only if  either  $a_0a_{m-1}=0$, or $a_ja_{j+1}=0$ for all $j=0,\dots, m-2$, or  $a_0=a_{m-1}$.
 }
 \end{remark}

 We denote $\sigma=\log m/\log n$.
We shall use ${\cal M}_{t,v,d}(n,m)$ to denote the class of
self-affine carpets in ${\cal M}_t(n,m)$ which possess vacant rows and whose uniform Bernoulli measures are doubling.
Using a kind of measure preserving property, \cite{Rao2019} proved that

\begin{prop}\label{pro:irrational}(\cite{Rao2019})  Let $\sigma \in \Q^c$, and  let  $E, F\in {\cal M}_{t,v,d}(n,m)$. If $E\sim F$,  then the distribution sequence of $E$ is a permutation of that of $F$.
\end{prop}

We say $E=K(n,m,\SD)$ has \emph{uniform horizontal fibers} if all non-zero $a_j$'s in the distribution sequence take only one value. It is  shown \cite{ Bed84, M84} that $\dim_H E=\dim_B E$ if and only if $E$ has uniform horizontal fibers. (In terms of Falconer \cite{Fa86},
a set is called \emph{regular} if its Hausdorff dimension and box dimension coincide.)
It is seen that if $E$ is a self-affine carpet with uniform horizontal fibers, then
the associated uniform Bernoulli measure must be doubling (see Remark \ref{rem:doubling}(ii)). We shall use ${\cal M}_{t,v,r}$ to denote
 the class of the self-affine carpets in ${\cal M}_{t,v,d}$ with uniform horizontal fibers.

Hence, the vacant row property, the doubling property and the uniform horizontal fibers property divide
the totally disconnected self-affine carpets into six subclasses, and if two self-affine carpets
are Lipschitz equivalent, then they must belong to the same subclass. (We shall use $\bar v$, $\bar d$ and $\bar r$ to denote the negation of the corresponding property.) See Figure \ref{fig:class}.

\begin{figure}
  \includegraphics[width=0.9\textwidth]{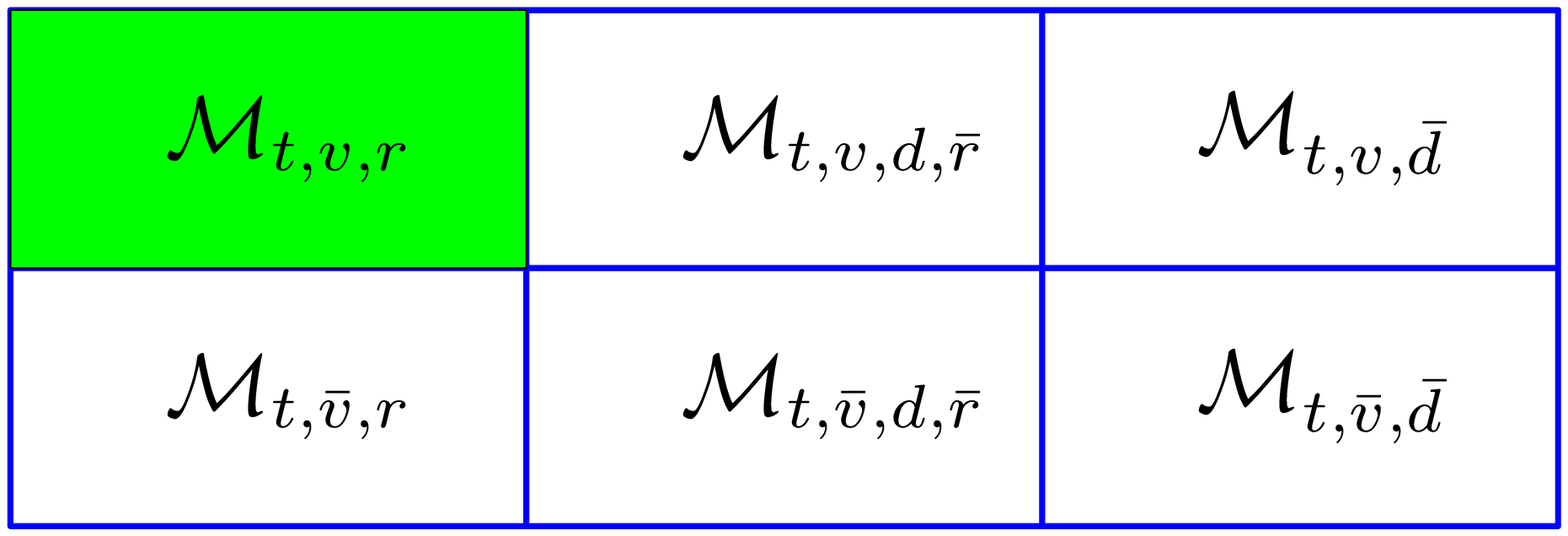}\\
  \caption{The collection ${\cal M}_t$ is divided into six subclasses.}
  \label{fig:class}
\end{figure}

The main goal of the present paper is to characterize the Lipschitz classification of self-affine carpets
in ${\cal M}_{t,v,r}$, the green part in Figure \ref{fig:class}. For this purpose, we use symbolic spaces.
For two sequences $\bx,\bx'\in \Z^\infty$, we use $\bx\wedge \bx'$ to denote their  maximal common prefix.
Let $\xi\in(0,1)$, we define a metric on $\Z^\infty$ by
\begin{equation}\label{eq:dxi}
d_{\xi}(\bx,\bx')=\xi^{|\bx\wedge\bx'|},
\end{equation}
where $|W|$ denotes the length of a word $W$.

Let us denote $\alpha=1/\sigma-\lfloor 1/\sigma \rfloor$ where $\lfloor x\rfloor$ is the greatest integer no larger than $x$.  Let   $s=\#\{j;~a_j>0\}$   be the number of non-vacant rows.
 For $k\geq 1$, define
\begin{equation}\label{eq:nk}
n_k=N s^{\lfloor 1/\sigma \rfloor+\delta_k-1}
\end{equation}
 where $\delta_k=\lfloor k\alpha\rfloor -\lfloor (k-1)\alpha\rfloor$, and set
$$
\Omega=\prod_{k=1}^\infty \{0,1,2,\dots, n_k-1\}.
$$

\begin{remark}{\rm  Notice that if $\sigma\in \Q^c$, then $(\delta_k)_{k\geq 1}$ is a \emph{sturmian sequence} with rotation $\alpha$ (for sturmian sequence, we refer to Chapter 2 of Lothaire \cite{Loth2002}); if $\sigma\in \Q$, then the sequence $(\delta_k)_{k\geq 1}$ is periodic.
}
\end{remark}

\begin{thm}\label{thm:sturmian} If $E=K(n,m,\SD)\in {\cal M}_{t,v,r}$, then
$E\sim (\Omega, d_{1/n})$ where $d_{1/n}$ is defined as \eqref{eq:dxi}. Especially, if $\sigma=p/q \in \Q$, then $E\sim (\{0,1,\dots, N^*-1\}^\infty, d_{1/n^p})$ where $N^*=N^ps^{q-p}$.
\end{thm}

Theorem \ref{thm:sturmian} says that if $\sigma\in \Q$, then $E\in {\cal M}_{t,v,r}$ is Lipschitz equivalent to a self-similar set, and if $\sigma\in \Q^c$, then $E$ is Lipschitz equivalent to a homogeneous Moran set.
(For  homogeneous Moran set, we refer to  Mauldin and Williams \cite{MW89}, and
Feng, Wen and Wu \cite{Feng}.)

To prove Theorem \ref{thm:sturmian},  we introduce a notion of the structure tree of
a metric space (Section 3).  The nested  structure of a set is an important tool to study  measures and dimensions. However, to study the Lipschitz equivalence, we need to handle the nested structures of two sets;
to overcome this difficulty, we regard a nested structure as a tree.

As a consequence of Proposition \ref{pro:irrational} (necessity) and  Theorem \ref{thm:sturmian} (sufficiency),
we have

\begin{thm}\label{thm:regular} Let  $E, F\in {\cal M}_{t,v,r}(n,m)$.  Then

$(i)$
If   $\sigma \in \Q$,  then $E\sim F$ if and only if   $\dim_H E=\dim_H F$;

$(ii)$ If   $\sigma \in \Q^c$,  then $E\sim F$ if and only if the distribution sequence of $E$ is a permutation of that of $F$.
\end{thm}

Our third result concerns another symbolization of  self-affine carpets.
We equip $\SD^\infty$ with the metric $\lambda$
defined as
\begin{equation}\label{eq:rho}
\lambda((\bx,\by), (\bx', \by'))=\max\{d_{1/n}(\bx, \bx'), d_{1/m}(\by, \by')\}.
\end{equation}

\begin{thm}\label{thm:symbol} If $E=K(n,m,\SD)\in {\cal M}_{t,v,r}$,  then $E \sim (\SD^\infty, \lambda).$
\end{thm}

It is  interesting  to know when  $K(n,m,\SD)\sim (\SD^\infty, \lambda)$.  Yang and Zhang\cite{YZ20} found several other classes
 of self-affine carpets with this property.

%
%

\medskip

The paper is organized as follows. In Section 2, we recall some known results about approximate squares of self-affine carpets. We define structure tree in Section 3 and develop several techniques
to handle structure trees in Section 4. In Section 5, we prove Theorems \ref{thm:sturmian} and \ref{thm:regular}.
Theorem \ref{thm:symbol} is proved in Section 6.

\section{\bf{Approximate squares  of  self-affine carpets}}\label{sec:square}

In this section, we study the structure of totally disconnected  self-affine carpet with vacant rows.
Let $E=K(n,m, \SD)$ be a self-affine carpet.
Denote
\begin{equation}\label{eq:row}
\SE=\{j;~a_j>0\}.
\end{equation}
 Throughout the paper, we will use the  notation
\begin{equation}\label{def:ell}
\ell(k)=\lfloor k/\sigma\rfloor.
\end{equation}

For $\bi=d_1\dots d_k\in \SD^k$, denote $S_{\bi}=S_{d_1}\circ \cdots \circ S_{d_k}$, and we call
 $S_{\bi}([0,1]^2)$ a \emph{basic rectangle of rank} $k$.
    Clearly,
$$S_{\bi}([0,1]^2)=S_{\bi}((0,0))+\left [ 0,\frac{1}{n^k}\right ]\times \left [ 0,\frac{1}{m^k}\right ]. $$
Set
$\widetilde {\mathbf E}_k=\bigcup_{\bi\in \SD^k} S_{\bi}([0,1]^2)$,
then $\widetilde {\mathbf E}_k$ decrease to $E$.

 Let $q\geq 2$ be an integer, and
$x_1\dots x_k\in \{0,1,\dots, q-1\}^k$, we will use the notation
$${0.x_1\dots x_k}|_q=\sum_{j=1}^k x_j q^{-j}.$$
Following
 McMullen \cite{M84}, we  divide a basic rectangle into approximate squares.

\begin{defi} {\rm (\textbf{Approximate squares})
Let $\bx=x_1\dots x_k\in \{0,1,\dots, n-1\}^k$ and $\by=y_1\dots y_{\ell(k)}\in \{0,1,\dots, m-1\}^{\ell(k)}$.
We call
\begin{equation}
{Q}(\bx,\by)=({0.\bx}|_n, {0.\by}|_m)+
 \left [0,\frac{1}{n^k}\right ]\times \left [0,\frac{1}{m^{\ell(k)}}\right ]
\end{equation}
  an  \emph{approximate square}  of rank $k$ of $E$, if $(x_j, y_j)\in \SD$ for $j\leq k$ and
$y_j\in \SE$ for $j>k$.
}
\end{defi}

An approximate square $Q'$ is called an \emph{offspring} of $Q$ if $Q'\subset Q$, and it is called a
\emph{direct offspring}
if the rank of $Q'$ equals the rank of $Q$ plus $1$. For $j\in\SE, $ we denote
$${\cal D}_j=\{i;(i,j)\in{\cal D}\}.$$
The following lemma is obvious, see \cite{Rao2019}.

 \begin{lemma}\label{lem:offspring}  (\cite{Rao2019}) Let $E=K(n,m,\SD)$, and  $Q(\bx, \by)$ be an approximate square of $E$ of rank $k$.

 (i) If $\ell(k)>k$, then  the direct offsprings of $Q(\bx, \by)$ are
 $$
\left \{Q(\bx \ast u, \by \ast \bz);~u\in \SD_{y_{k+1}} \text{ and } \bz\in \SE^{\ell(k+1)-\ell(k)}\right \}
$$
 and $Q(\bx, \by)$ has $a_{y_{k+1}} s^{\ell(k+1)-\ell(k)}$ direct offsprings;

 (ii) If $\ell(k)=k$, then the direct offsprings of $Q(\bx, \by)$ are
$$
\left \{Q(\bx \ast u, \by \ast v\ast  \bz);~ (u, v)\in \SD \text{ and } \bz  \in \SE^{\ell(k+1)-k-1}\right \},
$$
and $Q(\bx, \by)$ has $Ns^{\ell(k+1)-(k+1)}$ direct offsprings.
\end{lemma}


\begin{remark}\label{rem:cart} {\rm As a consequence of the above lemma, it is easy to see that

(i) The set of the direct offsprings of an approximate square of  $E=K(n,m,\SD)$   of rank $k$ can be written as
\begin{equation}\label{eq:cartesian}
A\times B+\left [0,\frac{1}{n^{k+1}}\right ]\times \left [0,\frac{1}{m^{\ell(k+1)}}\right ],
\end{equation}
where $A\subset  \Z/n^{k+1} $ and $B \subset \Z/m^{\ell(k+1)}$.

(ii) If $E=K(n,m,\SD)$ has uniform horizontal fibers, then the number of direct offsprings of  an approximate square of  $E$   of rank $k$ is always $Ns^{\ell(k+1)-\ell(k)-1}$ no matter $\ell(k)>k$ or not.
}
\end{remark}

\begin{lemma}\label{lem:measure} Let  $E$ be a self-affine carpet with uniform horizontal fibers. Let $W$  be an approximate square of rank $k$, then $\mu_E(W)=(N^{k}s^{\ell(k)-k})^{-1}$.
\end{lemma}

\begin{proof} Let $W=Q(\bx,\by)$ be an approximate square of rank $k$. Then the number of basic rectangles
of rank $\ell(k)$ contained in $Q(\bx,\by)$ is $(N/s)^{\ell(k)-k}$, and hence its measure in $\mu_E$
is $(N/s)^{\ell(k)-k}/N^{\ell(k)}$, the lemma is proved.
\end{proof}

Let ${\mathbf E}_k $ be the union of all approximate squares of rank $k$.
It is seen that $({\mathbf E}_k)_{k\geq 1}$ is a decrease sequence and   $E=\bigcap_{k=1}^\infty {\mathbf E}_k $.
Let $U$ be a connected component of $\mathbf E_k$.
 Hereafter, we will  call $U$  a \emph{component of $\mathbf E_k$} for simplicity. An approximate square of rank $k$ contained in $U$ will be called a \emph{member} of $U$.
  Denote by $\#_k(U)$ the number of members of $U$. It is shown that $\#_k(U)$ has an upper bound which
  is independent of $k$.

\begin{lemma}\label{lem:cartesian} (\cite{Rao2019})
Let $E=K(n,m ,{\cal D})$ be totally disconnected and possess  vacant rows. Then there exists $L_0>0$ such that
for every $k\geq 1$ and   every  component $U$ of $\mathbf E_k$, it holds that $\#_k(U)\leq  L_0$.
\end{lemma}

We shall denote by ${\mathcal C}_{E,k}$ the collection of components of $\mathbf E_k$, and set ${\mathcal C}_E=\bigcup_{k\geq 0} {\mathcal C}_{E,k}$, where we set $\mathbf E_0=[0,1]^2$ by convention.

\begin{remark}\label{lem:disconnected}{\rm  For self-affine carpets possessing vacant rows, there is a simple   criterion for totally disconnectedness (\cite{Rao2019}):
Let $E=K(n,m,{\mathcal D})$ and ${\mathcal D}$ possess vacant rows. Then $E$ is totally disconnected if and only if $ a_j<n$ for all $0\leq j\leq m-1$.
}
\end{remark}

\section{\textbf{Structure tree of a metric space}}\label{sec:tree}
 In this section, we introduce a structure tree   to describe the nested structure
 of a metric space.

\subsection{Tree and boundary} Let $T$ be a tree with a root and we denote the root by $\phi$.
Let  $v, v'$ be two vertices of $T$. The \emph{level} of $v$ is the distance from the root $\phi$ to $v$,
and we denote it by $|v|$.
 We say $v'$ is a \emph{direct offspring} of $v$, if there is an edge from $v$ to $v'$, and
 $|v'|=|v|+1$, and meanwhile, we say $v$ is the \emph{parent} of $v'$.
 We say $v'$ is an  offspring of $v$ if there is a path from $v$ to $v'$.

In this paper, we always assume that
any vertex of $T$ has at least one direct offspring, and the number of direct offsprings of a vertex is finite.
The \emph{boundary} of  $T$, denoted by $\partial T$, is defined to be
the collection of infinite path emanating from the root; we shall denote  such a path by $\bv=(v_k)_{k=0}^\infty$, where $v_k$ is a vertex of level $k$ in the path. In what follows, an infinite path always means that
a path emanating from the root.
Given $0<\xi<1$, we  equip  $ \partial T $ with the metric
$$
d_\xi(\bu,\bv)=\xi^{|\bu\wedge \bv|}.
$$

 \subsection{Structure trees of self-affine carpets.}
The following is an alternative description of a nested structure of a metric space.

\begin{defi}{\rm Let $X$ be a compact metric space and let  $T$ be a rooted tree.
 We call $T$ a \emph{structure tree} of $X$ if

 $(i)$ the root is $\phi=X$ and every vertex of $T$ is a closed subset of $X$;

 $(ii)$  the vertices of the same level are disjoint;

 $(iii)$ if  $\{v_1,\dots, v_p\}$ is the set of  direct offsprings of $v$, then $v=\bigcup_{j=1}^p v_j$.
 }
 \end{defi}

 \begin{example}\label{exam:self} {\rm Let $T$ be a rooted tree. Let $\partial T$ be the boundary of $T$ equipped with a metric $d_{\xi}$.
 If we identify  a vertex $v$ as the subset of $\partial T$ consisting of the infinite  paths passing $v$,
 then $T$ is a structure tree of $\partial T$.
 }
\end{example}

\begin{example}\label{exam:self} {\rm Here  we give two  structure trees of $E= K(n,m,{\mathcal D})$.

   \textbf{The first structure tree.} Let $T_1$ be a tree such that the  vertices  of level $k$ are  $W\cap E$, where $W$ runs over the approximate squares in ${\mathbf E}_k$. We set an edge from vertex $u$ to $v$
if $v\subset u$ and $|v|=|u|+1$.
 Clearly $T_1$ is a structure tree of $E$.

   \textbf{The second structure tree.} Let $T_2$ be a tree such that the vertices of  level $k$ are $U\cap E$, where $U$ runs over the components of ${\mathbf E}_k$.
 Then the  vertex set of $T_2$ is ${\mathcal C}_E\cap E=\bigcup_{k\geq 0}\left ( {\mathcal C}_{E,k}\cap E\right )$.  We define an edge from $u$ to $v$ if $v\subset u$ and $|v|=|u|+1$.
 We shall call $T_2$ the \emph{structure tree of $E$ induced by ${\mathcal C}_E$}.
 For simplicity, we sometimes say that  a component
 $U$ is a vertex of $T_2$ instead of $U\cap E$.
 }
\end{example}

\subsection{Regularity}
To study the Lipschitz classification,  we wish a structure tree has nice separation property. This motivates us to give the following definition.
Let $(X,d)$ be a metric  space. For two set $A, B\subset X$, we define
$d(A,B)=\inf\{d(a,b);~a\in A, b\in B\}.$

\begin{defi}{\rm Let $T$ be a structure tree of the compact metric space $(X,d)$.
If there exist a real number $0<\xi<1$ and a constant $\alpha_0>0$
such that,  for any $k\geq 0$ and any vertices $u,v$ of level $k$,
$$
\text{diam}~u\leq \alpha_0 \xi^k \ \text{ and } \
d(u, v)\geq \alpha_0^{-1} \xi^k,
$$
then we say $T$ is  $\xi$-regular.
}
\end{defi}

Similar idea has been appeared in Jiang, Wang and Xi \cite{Jiang19}, where a $\xi$-regular structure tree of $X$ is called a \emph{configuration} of $X$ there. The following result is essentially contained in \cite{Jiang19}.

\begin{theorem}\label{thm:Xi} Let $(X,d)$ be a compact metric space and
let $T$ be a structure tree of $X$.
If $T$ is $\xi$-regular, then  $(X,d)\sim (\partial T,d_\xi)$.
\end{theorem}

\begin{proof} For any $x\in X$, there is a unique infinite path $(v_k)_{k\geq 0}$
such that $\{x\}=\bigcap_{k\geq 0} v_k$.  Denote by $f$ the map from $\partial T$ to $X$ defined by
$f((v_k)_{k\geq 0})=x$.
We claim that $f$ is  bi-Lipschitz.
Pick  $(u_k)_{k\geq 0}, (v_k)_{k\geq 0}\in \partial T$,  and denote $x=f((u_k)_{k\geq 0}), y=f((v_k)_{k\geq 0})$.
Let $q$ be the length of common prefix of $(u_k)_{k\geq 0}$ and $(v_k)_{k\geq 0}$, then
$d_\xi((u_k)_{k\geq 0}, (v_k)_{k\geq 0})=\xi^q$  and   $x,y\in u_q.$
Since $T$ is $\xi$-regular,
we have
$$
\alpha_0^{-1}\xi^{q+1} \leq d_\xi(u_{q+1},v_{q+1})\leq d_{\xi}(x,y)\leq \text{diam}(u_q) \leq \alpha_0\xi^q,
$$
which implies that $f$ is bi-Lipschitz.
\end{proof}

\section{\textbf{$p$-Subtree and bundle map}}

In this section, we develop several techniques on structure tree.

\subsection{$p$-subtree}
Now we consider a special `subtree' of a rooted tree $T$.
Let $p\geq 2$ be an integer.
Let $T^*$ be the tree whose vertices of level $k$ consist of
  the vertices of $T$ of level $pk$, $k\geq 0$.
  For two vertices $u,v\in T^*$,  $v$
  is a direct offspring of $u$ if $v$ is a $p$-step offspring of $u$ in $T$ (that is, $v$ is an offspring of $u$ and $|v|=|u|+p$).
   We call $T^*$ the \emph{$p$-subtree} of $T$. The following result is obvious.

   \begin{theorem}\label{thm:subtree} For any $\xi\in (0,1)$, it holds that  $(\partial T, d_\xi)\sim (\partial T^*, d_{\xi^p})$.
   \end{theorem}

  \begin{proof} Let $f:\partial T\to \partial T^*$ be a map defined by $f((v_k)_{k\geq 0})=(v_{pk})_{k\geq 0}$.
   Clearly, $f$ is a bijection.
   Pick any $(u_k)_{k\geq 0},(v_k)_{k\geq 0}\in \partial T$, let $t$ be the length of the
    common prefix of $(u_{pk})_{k\geq 0}$ and $(v_{pk})_{k\geq 0}$,
    then
   $d_{\xi^p}(f((u_k)_{k\geq 0})),f((v_k)_{k\geq 0}))=\xi^{pt}$   and
   $\xi^{pt+p-1}\leq d_\xi( (u_k)_{k\geq 0}, (v_k)_{k\geq 0})\leq \xi^{pt},$
 so  $f$ is bi-Lipschitz.
  \end{proof}

   \subsection{bundle map} Let $T$ be a tree.  We call $B=\{v_1,\dots, v_t\}$ a \emph{bundle} of $T$ of level $k$, if $v_1,\dots, v_t$ are vertices of
$T$ of level $k$ and they sharing the same parent.

    Let $S$ and $T$ be two trees, and  let ${\cal C}_k$ and ${\cal C}'_k$  be the sets of vertices of level $k$ of them respectively.
    A map $\Delta$ defined on $\bigcup_{k\geq 0} {\cal C}_k$ is called a \emph{bundle map} from $S$ to $T$, if for each $k\geq 0$,
    it holds that

    $(i)$ For $u\in {\cal C}_k$, $\Delta(u)$ is a bundle of $T$ of level $k$;

    $(ii)$ $\{\Delta(u);~u\in{\cal C}_k\}$ is a partition of ${\cal C}'_k$;

    $(iii)$ If $u'$ is an offspring of $u$, then elements in $\Delta(u')$ are offsprings of elements in $\Delta(u)$.

 For a vertex $w$ of $T$, we use $[w]$ to denote the set of infinite paths in $\partial T$ passing $w$;
 if $B$ is a bundle of $T$, we denote $[B]=\bigcup_{w\in B}[w]$.

    \begin{theorem}\label{thm:bundle} Let $\xi\in(0,1)$, and let $S$ and $T$ be two trees. If there is a bundle map $\Delta$ from $S$ to $T$, then $(\partial S,d_{\xi})\sim (\partial T,d_{\xi})$.
    \end{theorem}

    \begin{proof}
  The map $\Delta$ induces a structure tree of $(\partial T, d_{\xi})$, which we will denote by $T^{*}$,  in the following way: the vertices of $T^{*}$ of level $k$ are
  $$\left \{[\Delta(u)];~ u\in {\cal C}_{k}\right \} .$$
 Clearly   $\Delta$ induces an isometry from $(\partial S,d_{\xi})$   to  $(\partial T^{*},d_{\xi})$.
  We claim that $T^*$ is $\xi$-regular.
Indeed, if $u$ and $v$ are vertices  of $S$ of level $k$, we have
$$
\text{diam}([\Delta(u)])\leq \xi^{k-1} \ \text{and} \ d_\xi([\Delta(u)],[\Delta(v)])\geq \xi^k.
$$
 So $(\partial T^{*},d_{\xi}) \sim (\partial T,d_{\xi})$ and the theorem is proved.
   \end{proof}

   \subsection{Homogeneous tree}
   Let $T$ be a tree with root $\phi$. If  every vertex of level $k-1$  has $n_k$ number of direct offsprings where $n_k\geq 1$, then
we call $T$ a \emph{homogeneous tree} with parameter $(n_k)_{k\geq 1}$.

A homogeneous tree can be regarded as a symbolic space.

\begin{lemma}\label{lem:product} Let $T$ be a homogeneous tree with parameters $(n_k)_{k\geq 1}$. Then
$$
(\partial T, d_\xi)\sim (\Omega, d_\xi),
$$
where $\Omega=\prod_{k=1}^\infty \{0, 1,\dots, n_k-1\}$.
\end{lemma}

\begin{proof}   Clearly we can label a vertex $u\in T$ of level $k$ as
$i_1\dots i_k\in \prod_{j=1}^k \{0,1,\dots, n_j-1\}$, and  the direct
offsprings of $i_1\dots i_k$ are $i_1\dots i_ki_{k+1}$ where $i_{k+1}\in \{0,1,\dots, n_{k+1}-1\}$.
Hence this labeling gives us an isometry between $\partial T$ and $\Omega$.
\end{proof}

\section{\textbf{Proof of Theorem \ref{thm:regular}}}\label{sec:regular}

In this section, we prove Theorem \ref{thm:regular}.
Let $E=K(n,m, \SD)\in {\mathcal M}_{t,v,r}$.
Recall that $N=\#\SD$ and $s=\#\SE$.
Then we have $a_j=0$ or $a_j=N/s$.
For $k\geq 1$, we define
\begin{equation}\label{eq:theta}
\theta_k=1/(N^k s^{\ell(k)-k}), \quad n_k=N s^{\ell(k)-\ell(k-1)-1},
\end{equation}
where we set  $\ell(0)=0$  by convention. (The $n_k$ defined above coincide with that in \eqref{eq:nk}.) Notice that $n_k\geq N$.
 Let us denote $\mu=\mu_E$ to be the uniform Bernoulli measure
of $E$.

Let  $W$ be an approximate square in ${\mathbf E}_k$. By Remark \ref{rem:cart},   $W$ has  $n_{k+1}$ direct offsprings.
Hence,  by Lemma \ref{lem:measure},
\begin{equation}\label{eq:mea_r_k}
\mu(W)=(n_1\cdots n_{k})^{-1}=\theta_k.
\end{equation}


\subsection{The coin lemma.}
The following lemma is motivated by Xi and Xiong \cite{XX10} which deals with the fractal cubes.
Recall that ${\mathcal C}_{E, k}$ is the collection of components of ${\mathbf E}_k$.

\begin{lemma}\label{lem:coin} \textbf{(Coin Lemma)} Let $k$ be an integer such that
$\ell(k)>k$. If $U\in {\mathcal C}_{E, k}$, then  there exist $V_1,\dots, V_q\in {\mathcal C}_{E,k+1}$  which are direct offsprings of $U$,
such that $\sum_{j=1}^q \mu(V_j)=\theta_k$.
\end{lemma}

\begin{proof}  We pick an $h\in \{0,1,\dots, m-1\}$ such
that    $a_{h}=0$; such $h$ exists since $E$ possesses vacant rows.
Let   $S_1,S_2,S_3$ and $S_4$  be the four approximate squares in  $U$ which locate
at the most top-left corner, the most top-right corner, the most bottom-left corner, and
the most bottom-right corner, respectively.
(We remark that  $S_i$ and $S_j$ may coincide.)

Write  $S_1=Q(\bx,\by)$  where $\by=y_1\dots y_{\ell(k)}$.
Recall that
${\mathcal D}_{y_{k+1}}=\{x;~ (x,y_{k+1})\in \SD\},$
 and its cardinality is less than $n$ by Remark \ref{lem:disconnected}.  Let $b_0\in \{0,1,\dots, n-1\}\setminus {\mathcal D}_{y_{k+1}}$, and denote by $z$    the left-bottom point of $S_1$.
Then by Remark \ref{rem:cart}(i), the horizontal lines passing $z+(0,\frac{h+0.5}{m^{\ell(k)+1}})$
and the vertical line passing $z+(\frac{b_0+0.5}{n^{k+1}},0)$ make a cross, and this cross
divides the direct offsprings of $S_1$ into four disjoint parts.
Especially, the offsprings in the left-top part are isolated and thus build  one or several components of ${\mathbf E}_{k+1}$; let us  denote the collection
of  these components by ${\cal U}_1$.

Denote  $S_2=Q(\bx',\by')$  and let $z'$ be the  left-bottom of
$S_2$. We have $\by=\by'$ since both of them are located in the top row of $U$, and it follows that
  $W$ is a direct offspring of $S_1$
if and only if $W+(z'-z)$ is a direct offspring of $S_2$.
Hence, shifting the above cross by $z'-z$, the new cross
divides the direct offsprings of $S_2$ into four disjoint parts, which are congruent to the four disjoint parts
of $S_1$ respectively.
Especially, the offsprings in the right-top part build one or  several components of ${\mathbf E}_{k+1}$, and we denote the collection
of  these components by ${\cal U}_2$. It is seen that
$$
\sum_{V\in {\cal U}_1\cup {\cal U}_2} \mu(V)=\frac{\#\{j;~a_j>0 \text{ and }j>h \}}{s} \cdot \theta_k.
$$

Similar as  above, there exist ${\cal U_3}$ and ${\cal U}_4$  which are
two collections of components of ${\mathbf E}_{k+1}$ contained in $S_3$ and $S_4$ respectively, such that
  $$
\sum_{V\in {\cal U}_3\cup {\cal U}_4} \mu(V)=\frac{\#\{j;~a_j>0 \text{ and }j<h\}}{s} \cdot \theta_k.
$$
(We remark that some of ${\cal U}_j, j=1,\dots, 4$, may be the empty set.) The lemma is proved in this case.
\end{proof}

\subsection{Homogeneous tree}
Let $T$ be a homogeneous tree  with parameter $(n_k)_{k\geq 1}$.
 Let $\nu$ be the uniform measure on $\partial T$.
Then for a vertex $u$ of $T$ of level $k$, we have
\begin{equation}\label{eq:mea_Homo}
\nu([u])=\left(\prod_{i=1}^{k} n_i \right)^{-1}=\theta_{k}.
\end{equation}

The following theorem asserts that  the tree $T$ is a symbolic representation of $E$.

\begin{theorem}\label{thm:homo} Let $E=K(n,m,\SD)\in {\cal M}_{t,v,r}$   and let $T$ be a homogeneous tree
with parameters $(n_k)_{k\geq 1}$ defined as \eqref{eq:theta}.
Then $E \sim (\partial T, d_{1/n})$.
\end{theorem}

\begin{proof}Let $S$ be the structure tree  of $E$ induced by the components in $\mathbf E_k$, $k\geq 0$.
(This is the second structure tree in Section 3.2.)
Let us denote the root of $S$ and $T$ by $\phi_S$ and $\phi_T$ respectively.
Clearly, $S$ is $1/n$-regular by Lemma \ref{lem:cartesian},  and hence $E \sim (\partial S, d_{1/n})$ by Theorem \ref{thm:Xi}.

 Let $L_0$ be the constant in Lemma \ref{lem:cartesian}, and  let $p$ be an integer such that
\begin{equation}\label{Np}
N^{p-1}\geq L_0^3 \ \text{ and } \ \ell(p-1)>p-1.
\end{equation}
Let $S^*$ and $T^*$ be the $p$-subtree of $S$ and $T$, respectively.
We shall construct a bundle map $\Delta$ from $S^*$ to $T^*$ such that $\Delta$ is measure preserving.

 First, we define $\Delta(\phi_S)=\phi_T$.
Suppose that $\Delta$ has been defined on $ \bigcup_{j=0}^k {\cal C}_{E,pj}$ already.
We are going to extend $\Delta$ to   ${\mathcal C}_{E, p(k+1)}$.

 Pick  $U\in{\cal C}_{E,pk}$. Let ${\cal V}=\{V_1,\dots, V_{r}\}$ be the set of offsprings of $U$ in
${\mathcal C}_{E,p(k+1)}$, that is, $V_j$'s are direct offsprings of $U$ in $S^*$.
Let $g$ be the number of  members of $U$; by Lemma \ref{lem:cartesian} we have $g\leq L_0$.
 We claim that

\medskip

\textit{Claim.  The collection ${\cal V}$ has a partition
\begin{equation}
{\cal V}={\cal V}'_1\cup \cdots \cup  {\cal V}'_g,
\end{equation}\label{eq:V}
such that  $\mu({\cal V}'_j)=\theta_{pk}$ for every $j=1,\dots, g$.}

\medskip

Let ${\cal U}=\{U_1,\dots, U_q\}$ be the set of  offsprings of $U$
 in ${\mathcal C}_{E, p(k+1)-1}$. Clearly
$\theta_{pk}\leq \mu(U)\leq qL_0\theta_{p(k+1)-1}$
and it follows that
$$q\geq \frac{\theta_{pk}}{\theta_{p(k+1)-1}}\frac{1}{L_0}\geq  \frac{ N^{p-1}}{L_0}\geq L_0^2.$$

Denote $\delta=\theta_{p(k+1)-1}$.
Since the rank of $U_i$ is no less than $p-1$ and $\ell(p-1)>p-1$, by Lemma \ref{lem:coin}, the direct offsprings of $U_i$ can be divided into two collections ${\cal V}^s_i$
and ${\cal V}^b_i$,  such that
the total measure of ${\cal V}^s_i$  is $\delta$ (the small collection),
and the total measure of ${\cal V}^b_i$ is $(N_i-1)\delta\leq (L_0-1)\delta$ (the residual collection), where $N_i$
is the number of members of $U_i$.
Therefore, we have a partition of ${\cal V}$ given by
$$
\bigcup_{i=1}^q \{{\cal V}_i^s,   {\cal V}_i^b\}.
$$

Now we regard $\delta$ as one `dollar' and regard each ${\cal V}^s_i$ as a one-dollar coin.
We regard ${\cal V}^b_i$ as a big coin that its value varies from $0$ to $L_0-1$. The total wealth of these coins is $gM$ `dollars', where $M=\prod_{j=1}^{p-1} n_{pk+j}$.
Then the claim  holds due to the following facts: first, the  value of every coin is no larger than $L_0-1$; secondly,
 we have plenty of  one-dollar coins (actually, the number is no less than $q$, and
   $q\geq L_0^2\geq gL_0$   since  $g\leq L_0$). Our claim is proved.

\medskip

By the induction hypothesis on $\Delta$, $\Delta(U)$ is a bundle
of $T^*$ of level $k$, which we write as
$$\Delta(U)=\{w_1,\dots, w_t \}.$$
  Clearly $t=g$,
since $\mu(U)=g\theta_{pk}$, $\nu([\Delta(U)])= t\theta_{pk}$,
and $\Delta$ is measure preserving.

 Let us  regard $\theta_{p(k+1)}$ as a 'cent'.  Then for any $V\in {\cal V}$,   $\mu(V)$ is a multiple of `cent',
 and the $\nu$-measure of a vertex of $T$ of level $p(k+1)$ is  one `cent'.

 For $j\in \{1,\dots, g\}$, ${\cal V}_j'$ is a collection of $p$-step offsprings of $U$ which we  write as
 $${\cal V}_j'=\{v_{j,1},\dots, v_{j,h_j}\};$$
 accordingly we take a partition
 $$
 {\cal W}_{j,1} \cup \cdots \cup  {\cal W}_{j,h_j}
 $$
of direct offsprings of $w_j$ in $T^*$  satisfying $\#{\cal W}_{j,i}=\mu(v_{j,i})/\theta_{p(k+1)}$. We define
\begin{equation}
\Delta(v_{j,i})={\cal W}_{j,i}, \quad j\in \{1,\dots, g\}, \quad i\in \{1,\dots, h_j\}.
\end{equation}
Repeating the above process for all $U$,
we extend $\Delta$ to  vertices of $S^*$ of level $k+1$.
It is not hard to check that  the three requirements in the definition of bundle map still hold and $\Delta$
is still measure preserving.

 Hence, $\Delta$ is a bundle map from $S^*$ to $T^*$, and   $(\partial S^*,d_{1/n^p})\sim (\partial T^{*},d_{1/n^p})$ by Theorem \ref{thm:bundle}.
On the other hand, by Theorem \ref{thm:subtree}, we have  $E \sim (\partial S,d_{1/n})\sim (\partial S^*, d_{1/n^p})$
and $(\partial T,d_{1/n})\sim (\partial T^*, d_{1/n^p})$, so
 \begin{equation}\label{eq:eqivalent_Homo}
E \sim  (\partial S^*,d_{1/n^p}) \sim (\partial T^{*},d_{1/n^p})\sim (\partial T,d_{1/n}).
\end{equation}
 The theorem is proved.
\end{proof}

\begin{proof}[\textbf{Proof of Theorem \ref{thm:sturmian}.}] Let $T$ be a homogeneous tree
with parameters $(n_k)_{k\geq 1}$.  By Theorem \ref{thm:homo} and Lemma \ref{lem:product}, we have
$$
E\sim (\partial T, d_{1/n})\sim \left (\prod_{k=1}^\infty \{0,1,\dots, n_k-1\}, d_{1/n}\right ),
$$
which proves the first assertion.

If   $\sigma=p/q\in \Q$, then  $\ell(pk)= qk$ for  $k\geq 0$. It follows that
$$
\prod_{j=1}^{p} n_{pk+j}=\left (\frac{N}{s}\right )^{p}\cdot s^{\ell(pk+p)-\ell(pk)}=N^p s^{q-p}.
$$
 Let  $T^*$ be the $p$-subtree of $T$ and denote  $N^*=N^p s^{q-p},$ then
$$E\sim (\partial T^*, d_{1/n^p})\sim (\{0,1,\dots, N^*-1\}^\infty,  d_{1/n^p} ).$$
 The second assertion is proved.
\end{proof}

Recall that the Hausdorff dimension of a self-affine carpet  is $\log_m\left (\sum_{j=0}^{m-1}a_j^\sigma\right )$ (\cite{Bed84, M84}).

\begin{proof}[\textbf{Proof of Theorem \ref{thm:regular}.}]
Let $N'$ be the cardinality of the digit set of $F$, and let $s'$ be the number of non-vacant rows of $F$.

(i) Let $\sigma=p/q\in \Q$.
 If $E$ and $F$ have the same Hausdorff dimension, then
$N^p  s^{q-p}=(N')^p (s')^{q-p}:=N^*$.
By Theorem \ref{thm:sturmian}, we have that both $E$ and $F$ are equivalent to the symbolic space
$(\{0,1,\dots, N^*-1\}^\infty,  d_{1/n^p} ), $
so $E\sim F$. That $E\sim F$ implies $\dim_H E=\dim_H F$ is folklore.
Assertion (i) is proved.

(ii) Let $\sigma\in {\mathbb Q}^c$.  If $E$ and $F$ share the same distribution sequence up to a permutation,
then $N=N'$ and $s=s'$.  By Theorem \ref{thm:sturmian},
$E$ and $F$ are Lipschitz equivalent to the same symbolic space $(\Omega, d_{1/n})$
  determined by $(n_k)_{k\geq 1}$ where
$n_k=N s^{\ell(k)-\ell(k-1)-1}$, so $E\sim F$. The necessity part is guaranteed by Proposition \ref{pro:irrational}. The second assertion is proved.
\end{proof}

\section{\textbf{Symbolic space}}
Recall that  $\lambda$ is a metric on  $\SD^\infty$
defined by
$$
\lambda((\bx, \by), (\bx', \by'))=\max\{d_{1/n}(\bx, \bx'), d_{1/m}(\by, \by')\}.
$$
\begin{proof}[\textbf{The proof of Theorem \ref{thm:symbol}. }]
First we define  a structure tree of $\SD^\infty$.
For two words $\bx$ and $\bi$, we denote  $\bx\lhd\bi$  if $\bx$ is a prefix of $\bi$.
For $k\geq 1$, we call
$$
[\bx,\by]=\{(\bi,\bj)\in \SD^\infty;~ \bx\lhd\bi, \by\lhd\bj\}
$$
an \emph{approximate square} of $\SD^\infty$ of rank $k$,  if $\bx=x_1\dots x_k, \by=y_1\dots y_{\ell(k)}$
are two words such that $(x_j,y_j)\in \SD$ for $1\leq j\leq k$ and $y_j\in\SE$ for $j>k$.
It is seen that  the approximate squares of the same rank are disjoint.

We denote by ${\cal C}_k$  the collection of approximate squares of rank $k$ of $\SD^\infty$, especially we set ${\cal C}_0=\SD^\infty$ by convention.
Let $T$ be  the structure  tree of $\SD^\infty$ induced by $({\cal C}_k)_{k\geq 0}$.

We claim that $ T$ is $1/n$-regular.
Indeed, if $[\bx,\by]$ and $[\bx',\by']$ are two approximate squares of rank $k$,
an easy calculation shows that $\lambda([\bx,\by],[\bx',\by'])\geq 1/n^k$ and
$$
  \text{diam}([\bx,\by])\leq \max\{1/n^k, 1/m^{\ell(k)}\}\leq m/n^k.
$$
It follows that $T$ is $1/n$-regular, so by Theorem \ref{thm:Xi}, we have
$(\partial T,d_{1/n})\sim (\SD^\infty, \lambda)$.

On the other hand,
 every vertex $[\bx,\by]$ of $T$ of level $k$ has $n_{k+1}:=Ns^{\ell(k+1)-\ell(k)-1}$ direct offsprings (see Remark \ref{rem:cart}), so $ T$ is a homogeneous tree
with parameters $(n_k)_{k\geq 1}$.
Therefore,  $ (\SD^\infty, \lambda) \sim (\partial T,d_{1/n})\sim E$  where the last equivalence is due to  Theorem \ref{thm:homo}.
\end{proof}


\end{document}